\documentclass[12pt,twoside,reqno]{amsart}
\usepackage{amsxtra,amscd}
\usepackage{graphicx}
\usepackage{amsmath}
\usepackage{amsfonts}
\usepackage{amssymb}

\newtheorem{theorem}{Theorem}[section]
\newtheorem{lemma}[theorem]{Lemma}
\theoremstyle{definition}
\newtheorem{definition}[theorem]{Definition}

\newtheorem{remark}[theorem]{Remark}

\newtheorem{corollary}[theorem]{Corollary}

\numberwithin{equation}{section}

\begin{document}
\title{on the arithmetic fundamental groups}
\author{Feng-Wen An}
\address{School of Mathematics and Statistics, Wuhan University, Wuhan,
Hubei 430072, People's Republic of China}
\email{fwan@amss.ac.cn}

\begin{abstract}
In this paper we will define a qc fundamental group for an arithmetic scheme by quasi-galois closed covers. Then we will give a computation for such a group and will prove that the \'{e}tale fundamental group of an arithmetic scheme is a normal subgroup in our qc fundamental group, which make up the main theorem of the paper. Hence, our group gives us a prior estimate of the \'{e}tale fundamental group. The quotient group reflects the topological properties of the scheme.
\end{abstract}

\maketitle

\begin{center}
{\tiny {Contents} }
\end{center}

{\tiny \qquad {Introduction} }

{\tiny \qquad {1. Definition and Notation} }

{\tiny \qquad {2. Statement of The Main Theorem} }

{\tiny \qquad {3. Proof of The Main Theorem}}

{\tiny \qquad {References}}

\section*{Introduction}

One  has been
 used the related data of arithmetic varieties $X/Y$ to describe
a given Galois extension $E/F$ for many years (for example, see \cite
{Bloch,Kato,Schmidt,Raskind,Saito,VS1,w1,w2}). It has been seen that there is a nice relationship
 between the Galois group $Gal(E/F)$ and the automorphism group $Aut(X/Y)$ for the case that $E/F$ are canonically the function fields $k(X)/k(Y)$.

 Here, one says that the arithmetic varieties $X/Y$ are a
\emph{geometric model} for the Galois extensions $E/F$ if the Galois group $
Gal\left( E/F\right) $ is isomorphic to the automorphism group ${Aut}\left(
X/Y\right) $ (for example, see \cite{SGA1,Raskind,VS1,SV2}).

Among the whole of the invariants on arithmetic schemes, it has been seen that the \'{e}tale fundamental groups are a very important tool for one to obtain class fields since these groups encode the
whole of the
information of class fields of the number fields or function fields. So, it is a natural task how to compute the \'{e}tale fundamental groups.

In this paper we will define a qc fundamental groups for an arithmetic scheme
 by quasi-galois closed covers in a manner similar to an \'{e}tale fundamental group of the scheme.
Here, as behave like Galois extensions, the quasi-galois closed covers have many desired properties of the arithmetic schemes (for example, see \cite{An2,An3}).

In deed, quasi-galois closed covers can be
regarded as a generalization of the pseudo-galois covers of arithmetic
varieties in the sense of Suslin-Voevodsky (see \cite{VS1,SV2}).

In \emph{Theorem 2.2}, the main theorem of the paper,
we will give a computation for  a qc fundamental group of an arithmetic scheme; we will also prove that the \'{e}tale fundamental group of the scheme is a normal subgroup in our qc fundamental group. Hence, by our groups we will obtain a prior estimate of the \'{e}tale fundamental groups.

On the other hand, their quotient group reflect the topological properties of the scheme. It will be seen that an arithmetic scheme has no finite branched cover if and only if the quotient group is trivial.

\textbf{Acknowledgment.} The author would like to express his sincere
gratitude to Professor Li Banghe for his invaluable advice and instructions
on algebraic geometry and topology.

\section{Definition and Notation}

\subsection{Convention}

In this paper, an \textbf{arithmetic variety}
is an integral scheme surjectively over $Spec\left( \mathbb{Z}\right) $ of finite type; an \textbf{integral variety} is an  integral
scheme surjectively over $Spec\left( \mathbb{Z}\right) $.
Let $k(X)$ denote the function field of an integral scheme $X$.

For an integral domains $D$, we will let $Fr(D)$ denote the field of
fractions on $D$. In particular, the field $Fr(D)$ will always assumed to be
contained in $\Omega $ if $D$ is a subring of a field $\Omega $.

For a field $K,$ we let $K^{al}$ denote the algebraically closed closure of $%
K.$ Define $K^{un}$ to be the union of all the finite unramified
subextensions $E$ ($\subseteq K^{al}$) over $K.$

For an arithmetic variety $X$, let $\pi _{1}^{et}\left( X;s\right) $ denote
the \'{e}tale fundamental group of $X$ for a given geometric point $s$ over $%
X$ (for detail see \cite{F-K,SGA1,Mln}).

\subsection{Affine covering with values in a field}

For convenience, let us fix notation and recall some definitions in the
paper (for details, see \cite{An1,An2,An3}). Fixed a scheme $(X,\mathcal{O}%
_{X})$.

An \textbf{affine covering} of $(X,\mathcal{O}_{X})$ is a family $\mathcal{C}%
_{X}=\{(U_{\alpha },\phi _{\alpha };A_{\alpha })\}_{\alpha \in \Delta }$
such that $\phi _{\alpha }$ is an isomorphism from an open set $U_{\alpha }$
of $X$ onto the spectrum $Spec{A_{\alpha }}$ of a commutative ring $%
A_{\alpha }$ for any $\alpha \in \Delta $. Each $(U_{\alpha },\phi _{\alpha
};A_{\alpha })\in \mathcal{C}_{X}$ is called a \textbf{local chart}.
Moreover, $\mathcal{C}_{X}$ is said to be \textbf{reduced} if $U_{\alpha
}\neq U_{\beta }$ holds for any $\alpha \neq \beta $ in $\Delta $.

In particular, an \textbf{affine patching} of $(X,\mathcal{O}_{X})$ is an
affine covering $\{(U_{\alpha },\phi _{\alpha };A_{\alpha })\}_{\alpha \in
\Delta }$ of $(X,\mathcal{O}_{X})$ such that $\phi _{\alpha }$ is the
identity map on $U_{\alpha }=SpecA_{\alpha }$ for each $\alpha \in \Delta .$
Evidently, an affine patching is reduced.

Let $\mathfrak{Comm}$ be the category of commutative rings with identity.
Fixed a subcategory $\mathfrak{Comm}_{0}$ of $\mathfrak{Comm}$. An affine
covering $\{(U_{\alpha},\phi_{\alpha};A_{\alpha })\}_{\alpha \in \Delta}$ of
$(X, \mathcal{O}_{X})$ is said to be \textbf{with values} in $\mathfrak{Comm}%
_{0}$ if
 for each $\alpha \in \Delta $ there are $\mathcal{O}_{X}(U_{\alpha})=A_{\alpha}$ and $U_{\alpha}=Spec(A_{\alpha})$, where
 $A_{\alpha }$ is a ring contained in $\mathfrak{Comm}_{0}$.

Let $\Omega $ be a field and let $\mathfrak{Comm}(\Omega )$ be the category
consisting of the subrings of $\Omega $ and their isomorphisms. An affine
covering $\mathcal{C}_{X}$ of $(X,\mathcal{O}_{X})$ with values in $%
\mathfrak{Comm}(\Omega )$ is said to be \textbf{with values in the field $%
\Omega $}.

Let $\mathcal{O}_{X}$ and $\mathcal{O}^{\prime}_{X}$ be two structure sheaves on the underlying space of an integral scheme $X$. The two integral schemes $(X,\mathcal{O}_{X})$ and $(X, \mathcal{O}^{\prime}_{X})$ are said to be \textbf{essentially equal} provided that for any open set $U$ in $X$, we have
 $$U \text{ is affine open in }(X,\mathcal{O}_{X}) \Longleftrightarrow \text{ so is }U \text{ in }(X,\mathcal{O}^{\prime}_{X})$$ and in such a case,  $D_{1}=D_{2}$ holds or  there is $Fr(D_{1})=Fr(D_{2})$ such that for any nonzero $x\in Fr(D_{1})$, either $$x\in D_{1}\bigcap D_{2}$$ or $$x\in D_{1}\setminus D_{2} \Longleftrightarrow x^{-1}\in D_{2}\setminus D_{1}$$ holds, where $D_{1}=\mathcal{O}_{X} (U)$ and $D_{2}=\mathcal{O}^{\prime}_{X} (U)$.

 Two schemes $(X,\mathcal{O}_{X})$ and $(Z,\mathcal{O}_{Z})$ are said to be \textbf{essentially equal} if the underlying spaces of $X$ and $Z$ are equal and the schemes $(X,\mathcal{O}_{X})$ and $(X,\mathcal{O}_{Z})$ are essentially equal.

\subsection{Galois extensions}

Let $K$ be a finitely generated extension of a field $k$. Here $K$ is not
necessarily algebraic over $k$. As usual, $Gal\left( K/k\right) $ denotes
the Galois group of $K$ over $k$.

The field $K$ is said to be a \textbf{Galois extension} of $k$ if $k$ is the
fixed subfield of the Galois group $Gal\left( K/k\right) $ in $K$.

\subsection{Quasi-galois closed}

Let $X$ and $Y$ be two integral varieties and let $f:X\rightarrow Y$ be a
surjective morphism. Denote by $Aut\left( X/Y\right) $ the group of
automorphisms of $X$ over $Z$.

By a \textbf{conjugate} $Z$ of $X$ over $Y$, we will understand an integral
variety $Z$ that is isomorphic to $X$ over $Y$.

\begin{definition}
$X$ is said to be \textbf{quasi-galois closed} over $Y$ by $f$ if there is an algebraically closed field $\Omega$
and a reduced affine covering $\mathcal{C}_{X}$ of $X$ with values in $
\Omega $ such that for any conjugate $Z$ of
$X$ over $Y$ the two conditions are satisfied:
\begin{itemize}
\item $(X,\mathcal{O}_{X})$ and $(Z,\mathcal{O}_{Z})$ are essentially equal if $Z$ has a reduced
affine covering with values in $\Omega$.

\item $\mathcal{C}_{Z}\subseteq \mathcal{C}_{X}$ holds if $\mathcal{C}_{Z}$
is a reduced affine covering of $Z$ with values in $\Omega $.
\end{itemize}
\end{definition}

\begin{remark}
It is seen that a quasi-galois closed variety $X$ has only one conjugate in
the algebraically closed field $\Omega =k\left( X\right) ^{al}$. That is,
let $Z$ and $Z^{\prime }$ be conjugates of $X.$ Then  $Z$ and $Z^{\prime }$ must be essentially equal if $Z$ and $Z^{\prime }$ both have reduced affine coverings
with values in $\Omega .$
\end{remark}

\section{Statement of The Main Theorem}

\subsection{Definition for qc fundamental group}

Let $X$ be an arithmetic variety. Fixed an algebraically closed field $%
\Omega $ such that the function field $k\left( X\right) $ is contained in $%
\Omega .$ Here, $\Omega $ is not necessarily algebraic over $k\left(
X\right) .$

Define $X_{qc}\left[ \Omega \right] $ to be the set of arithmetic varieties $%
Z$ satisfying the following conditions: $\left( i\right) $ $Z$ has a reduced
affine covering with values in $\Omega $; $\left( ii\right) $ there is a
surjective morphism $f:Z\rightarrow X$ of finite type such that $Z$ is
quasi-galois closed over $X.$

Set a partial order $\leq$ in the set $X_{qc}\left[ \Omega \right] $ in such
a manner:

Take any $Z_{1},Z_{2}\in X_{qc}\left[ \Omega \right] ,$ we say
\begin{equation*}
Z_{1}\leq Z_{2}
\end{equation*}
if there is a surjective morphism $\varphi :Z_{2}\rightarrow Z_{1}$ of
finite type such that $Z_{2}$ is quasi-galois closed over $Z_{1}.$

By \emph{Lemmas 3.6,3.8-10} below, it is seen that $X_{qc}\left[ \Omega %
\right] $ is a directed set and
\begin{equation*}
\{Aut\left( Z/X\right) :Z\in X_{qc}\left[ \Omega \right] \}
\end{equation*}%
is an inverse system of groups. Hence, we have the following definition.

\begin{definition}
Let $X$ be an arithmetic variety and let $\Omega $ be an algebraically
closed field $\Omega $ such that $\Omega \supseteq k\left( X\right) .$ The
inverse limit
\begin{equation*}
\pi _{1}^{qc}\left( X;\Omega \right) \triangleq {\lim_{\longleftarrow}}
_{Z\in X_{qc}\left[ \Omega \right] }{Aut\left( Z/X\right)}
\end{equation*}
of the inverse system $\{Aut\left( Z/X\right) :Z\in X_{qc}\left[ \Omega %
\right] \}$ of groups is said to be the \textbf{qc fundamental group} of the
scheme $X$ with coefficient in $\Omega .$
\end{definition}

\subsection{Statement of the main theorem}

The following is the main theorems of the paper.

\begin{theorem}
Let $X$ be an arithmetic variety. Take any algebraically closed field $%
\Omega $ such that $\Omega \supseteq k\left( X\right) .$ Then we have the
following statements.

$\left( i\right) $ There is a group isomorphism
\begin{equation*}
\pi _{1}^{qc}\left( X;\Omega \right) \cong Gal\left( {\Omega }/k\left(
X\right) \right) .
\end{equation*}

$\left( ii\right) $ Take any geometric point $s$ of $X$ over $\Omega $. Then
there is a group isomorphism
\begin{equation*}
\pi _{1}^{et}\left( X;s\right) \cong \pi _{1}^{qc}\left( X;\Omega \right)
_{et}
\end{equation*}%
where $\pi _{1}^{qc}\left( X;\Omega \right) _{et}$ is a subgroup of $\pi
_{1}^{qc}\left( X;\Omega \right) $. Moreover, $\pi _{1}^{qc}\left( X;\Omega
\right) _{et}$ is a normal subgroup of $\pi _{1}^{qc}\left( X;\Omega \right)
$.\bigskip
\end{theorem}

\begin{remark}
Let $X$ be an arithmetic variety. Put $$\pi _{1}^{qc}\left( X \right)=\pi _{1}^{qc}\left( X;{k(X)}^{al} \right).$$ Then there is a group
isomorphism
$$
\pi _{1}^{qc}\left( X \right) \cong Gal\left( {k(X)}^{al}/k\left( X\right)
\right).
$$
\end{remark}

\begin{definition}
Let $X$ be an arithmetic variety. The quotient group
\begin{equation*}
\pi _{1}^{br}\left( X\right) =\pi _{1}^{qc}\left( X;k\left( X\right)
^{al}\right) /\pi _{1}^{qc}\left( X;k\left( X\right) ^{al}\right) _{et}
\end{equation*}
is said to be the \textbf{branched group} of the arithmetic variety $X$.
\end{definition}

The branched group $\pi _{1}^{br}\left( X\right)$ can reflect the
topological properties of the scheme $X,$ especially the properties of the
associated complex space $X^{an}$ of $X,$ for example, the branched covers
of $X^{an}$. In fact, we have the following corollary.

\begin{corollary}
Let $X$ be an arithmetic variety. Then we have
\begin{equation*}
\pi _{1}^{br}\left( X\right)=\{0\}
\end{equation*}
if and only if $X$ has no finite branched cover.
\end{corollary}

\begin{proof}
Trivial.
\end{proof}

\section{Proof of The Main Theorem}

\subsection{Recalling definitions and basic facts}

Let us recall some definitions and basic facts about quasi-galois closed
schemes (see \cite{An2,An3,An4}).

Let $K$ be a finitely generated extension of a field $k$. Here $K$ is not
necessarily algebraic over $k$.

\begin{definition}
(see \cite{An2}) $K$ is \textbf{quasi-galois} over $k$ if each irreducible
polynomial $f(X)\in F[X]$ that has a root in $K$ factors completely in $K%
\left[ X\right] $ into linear factors for any intermediate field $k\subseteq
F\subseteq K$.
\end{definition}

We call the elements $t_{1},t_{2},\cdots ,t_{n}\in K\setminus k $ a \textbf{$%
(r,n)-$nice basis} of $K$ over $k$ if the conditions are satisfied:

$\left( i\right) $ $L=K(t_{1},t_{2},\cdots ,t_{n})$;

$\left( ii\right) $ $t_{1},t_{2},\cdots ,t_{r}$ constitute a transcendental
basis of $L$ over $K$;

$\left( iii\right) $ $t_{r+1},t_{r+2},\cdots ,t_{n}$ are linearly
independent over $K(t_{1},t_{2},\cdots ,t_{r})$, where $0\leq r\leq n$.

\begin{definition}
(see \cite{An2}) Let $D\subseteq D_{1}\cap D_{2}$ be three integral domains.
The ring $D_{1}$ is said to be \textbf{quasi-galois} over $D$ if the field $%
Fr\left( D_{1}\right) $ is a quasi-galois extension of $Fr\left( D\right) $.
\end{definition}

\begin{definition}
(see \cite{An2}) The ring $D_{1}$ is a \textbf{conjugation} of $D_{2}$ over $%
D$ if there is a $(r,n)-$nice $k-$basis $w_{1},w_{2},\cdots ,w_{n}$ of the
field $Fr(D_{1})$ and an $F-$isomorphism $\tau _{(r,n)}:Fr(D_{1})\rightarrow
Fr(D_{2})$ of fields such that
\begin{equation*}
\tau _{(r,n)}(D_{1})=D_{2},
\end{equation*}%
where $k=Fr(D)$ and $F\triangleq k(w_{1},w_{2},\cdots ,w_{r})$ is assumed to
be contained in the intersection $Fr(D_{1})\cap Fr(D_{2})$.
\end{definition}

Let $X$ and $Y$ be two integral varieties and let $\varphi :X\rightarrow Y$
be a surjective morphism. Fixed an algebraically closed closure $\Omega $ of
the function field $k\left( X\right) $.

\begin{definition}
(see \cite{An3}) A reduced affine covering $\mathcal{C}_{X}$ of $X$ with
values in $\Omega $ is said to be \textbf{quasi-galois closed} over $Y$ by $%
\varphi$ if the below condition is satisfied:

There exists a local chart $(U_{\alpha }^{\prime },\phi _{\alpha }^{\prime
};A_{\alpha }^{\prime })\in \mathcal{C}_{X}$ such that $U_{\alpha }^{\prime
}\subseteq \varphi ^{-1}(V_{\alpha })$ for any $(U_{\alpha },\phi _{\alpha
};A_{\alpha })\in \mathcal{C}_{X}$, for any affine open set $V_{\alpha }$ in
$Y$ with $U_{\alpha }\subseteq \varphi ^{-1}(V_{\alpha })$, and for any
conjugation $A_{\alpha }^{\prime }$ of $A_{\alpha }$ over $B_{\alpha }$,
where $B_{\alpha }$ is the canonical image of $\mathcal{O}_{Y}(V_{\alpha })$
in the function field $k(Y)$.
\end{definition}

We have the following lemmas for the criterion of quas-galois closed.

\begin{lemma}
\emph{(c.f. \cite{An3})} Assume that $k(Y)$ is contained in $\Omega $. Then $%
X$ is quasi-galois closed over $Y$ if there is an affine patching $\mathcal{C%
}_{X}$ of $X$ with values in $\Omega $ such that

\begin{itemize}
\item either $\mathcal{C}_{X}$ is quasi-galois closed over $Y$,

\item or $A_{\alpha }$ has only one conjugate over $B_{\alpha}$ for any $%
(U_{\alpha },\phi _{\alpha };A_{\alpha })\in \mathcal{C}_{X}$ and for any
affine open set $V_{\alpha}$ in $Y$ with $U_{\alpha }\subseteq
\varphi^{-1}(V_{\alpha})$, where $B_{\alpha}$ is the canonical image of $%
\mathcal{O}_{ Y}(V_{\alpha})$ in the function field $k(Y)$.
\end{itemize}
\end{lemma}

\begin{proof}
It is immediate from definition.
\end{proof}

We have the lemma below for the existence of quasi-galois closed.

\begin{lemma}
\emph{(see \cite{An3})} Let $K$ be a finitely generated extensions of a
number field and let $Y$ be an arithmetic variety with $K=k\left( Y\right) $%
. Fixed any finitely generated extensions $L$ of $K$ such that $L$ is Galois
over $K$.

Then there exists an arithmetic variety $X$ and a surjective morphism $%
f:X\rightarrow Y$ of finite type such that

\begin{itemize}
\item $L=k\left( X\right) $;

\item the morphism $f$ is affine;

\item $X$ is a quasi-galois closed over $Y$ by $f$.
\end{itemize}
\end{lemma}

We have the lemma below for the property of quasi-galois closed.

\begin{lemma}
\emph{(see \cite{An2})} Let $X$ and $Y$ be two arithmetic varieties. Assume
that $X$ is quasi-galois closed over $Y$ by a surjective morphism $\phi$ of
finite type.

Then the function field $k\left( X\right) $ is canonically a Galois
extension of $k(Y)$ and there is a group isomorphism
\begin{equation*}
{Aut}\left( X/Y\right) \cong Gal(k\left( X\right) /k(Y)).
\end{equation*}

Moreover, let $\dim X=\dim Y$. Then $\phi$ is finite and $X$ is a
pseudo-galois cover of $Y$ in the sense of Suslin-Voevodsky.
\end{lemma}

At last we have the useful lemma below.

\begin{lemma}
\emph{(c.f. \cite{An4})} Let $X$ be an integral variety. Then there is an
integral variety $Z$ satisfying the conditions:

\begin{itemize}
\item $k\left( X\right) =k\left( Z\right) \subseteq \Omega ;$

\item $X\cong Z$ are isomorphic;

\item $Z$ has a reduced affine covering with values in $\Omega .$
\end{itemize}
\end{lemma}

\begin{proof}
Trivial.
\end{proof}

\subsection{Recalling the construction for a geometric model}

(see \cite{An3,An4}) Let $Y$ be an arithmetic variety. Let $L$ be a Galois
extension of the function field $K=k\left( Y\right) .$ In the following we
will proceed in several steps to construct a scheme $X$ and a surjective
morphism $f:X\rightarrow Y$ such that $L=k\left( X\right) $ and $X$ is
quasi-galois closed over $Y$ by $f$.

By \emph{Lemma 3.8}, without loss of generality, assume that there is a
reduced affine covering $\mathcal{C}_{Y}$ of the scheme $Y$ with values in $%
\Omega .$

\emph{\textbf{Step 1.}} Fixed a set $\Delta $ of the generators $L$ over $K.$
That is, $L=K\left( \Delta \right) $ and $\Delta \subseteq L\setminus K.$
Put $G=Gal\left( L/K\right) .$

\emph{\textbf{Step 2.}} Take any local chart $\left( V,\psi
_{V},B_{V}\right) \in \mathcal{C}_{Y}.$ Define $A_{V}$ to be the subring of $%
L$ generated over $B_{V}$ by the set
\begin{equation*}
\Delta _{V}=\{\sigma \left( x\right) \in L:\sigma \in G,x\in \Delta \}.
\end{equation*}

We have
\begin{equation*}
A_{V}=B_{V}\left[ \Delta _{V}\right] .
\end{equation*}

Set
\begin{equation*}
i_{V}:B_{V}\rightarrow A_{V}
\end{equation*}%
to be the inclusion.

\emph{\textbf{Step 3.}} Define
\begin{equation*}
\Sigma =\coprod\limits_{\left( V,\psi _{V},B_{V}\right) \in \mathcal{C}%
_{Y}}Spec\left( A_{V}\right)
\end{equation*}
to be the disjoint union and define
\begin{equation*}
\pi _{Y}:\Sigma \rightarrow Y
\end{equation*}%
to be the projection.

Then $\Sigma $ is a topological space, where the topology $\tau _{\Sigma }$
on $\Sigma $ is naturally determined by the Zariski topologies on all $%
Spec\left( A_{V}\right) .$

\emph{\textbf{Step 4.}} Define an equivalence relation $R_{\Sigma}$ in $%
\Sigma$ in such a manner:

For any $x_{1},x_{2}\in \Sigma $, we say
$
x_{1}\sim x_{2}
$
if and only if
$
j_{x_{1}}=j_{x_{2}}
$
holds in $L$. Here $j_{x}$ denotes the corresponding prime ideal of $A_{V}$ to a point $%
x\in Spec\left( A_{V}\right) $ (see \cite{EGA}).

Define
\begin{equation*}
X=\Sigma /\sim .
\end{equation*}
Let
\begin{equation*}
\pi _{X}:\Sigma \rightarrow X
\end{equation*}
be the projection.

\emph{\textbf{Step 5.}} Define a map
\begin{equation*}
f:X\rightarrow Y
\end{equation*}%
by
\begin{equation*}
\pi _{X}\left( z\right) \longmapsto \pi _{Y}\left( z\right)
\end{equation*}%
for each $z\in \Sigma $.

\emph{\textbf{Step 6.}} Put
\begin{equation*}
\mathcal{C}_{X}=\{\left( U_{V},\varphi _{V},A_{V}\right) \}_{\left( V,\psi
_{V},B_{V}\right) \in \mathcal{C}_{Y}}
\end{equation*}%
where $U_{V}=\pi _{Y}^{-1}\left( V\right) $ and $\varphi
_{V}:U_{V}\rightarrow Spec(A_{V})$ is the identity map for each $\left(
V,\psi _{V},B_{V}\right) \in \mathcal{C}_{Y}$.

Define the scheme
$
\left( X,\mathcal{O}_{X}\right)
$
to be obtained by gluing the affine schemes $Spec\left( A_{V}\right) $ for
all local charts $\left( V,\psi _{V},B_{V}\right) \in \mathcal{C}_{Y}$ with
respect to the equivalence relation $R_{\Sigma }$ (see \cite{EGA,Hrtsh}).

Then $\left( X,\mathcal{O}_{X}\right) $ is the desired scheme and $%
f:X\rightarrow Y$ is the desired morphism of schemes. This completes the
construction.

\subsection{Any two qc varieties have a common qc cover}

Let $X$ be an arithmetic variety and let $\Omega $ be an algebraically
closed field containing $k\left( X\right) .$ We need the following lemma in
order to prove that $X_{qc}\left[ \Omega \right] $ is a directed set.

\begin{lemma}
Take any $Z_{1},Z_{2}\in X_{qc}\left[ \Omega \right] .$ There is a third $%
Z_{3}\in X_{qc}\left[ \Omega \right] $ such that $Z_{3}$ is quasi-galois
closed over $Z_{1}$ and $Z_{2},$ respectively.
\end{lemma}

\begin{proof}
Let $f_{1}:Z_{1}\rightarrow X$ and $Z_{2}\rightarrow X$ be two surjective
morphisms of finite types such that $Z_{1}/X$ and $Z_{2}/X$ are quasi-galois
closed. It is seen that $f_{1}$ and $f_{2}$ are both affine by \emph{%
Definition 1.1} and \emph{Lemma 3.6.} And it is seen that the function
fields $k\left( Z_{1}\right) /k\left( X\right) $ and $k\left( Z_{2}\right)
/k\left( X\right) $ are both Galois from \emph{Lemma 3.7}.

Without loss of
generality, by \emph{Lemma 3.8}, assume that $X$ has a reduced affine
covering $\mathcal{C}_{X}$ with values in $\Omega .$

Fixed any affine chart $V$ contained in $\mathcal{C}_{X}.$ Put $$\Delta
=\Delta _{1}\cup \Delta _{2}$$ and $$G=Gal\left( k\left( Z_{1}\right) \cdot
k\left( Z_{2}\right) /k\left( X\right) \right) .$$ Here, $\Delta _{j}$ is the
set of generators of the ring $\mathcal{O}_{Z_{j}}\left( f_{j}^{-1}\left(
V\right) \right) $ over $\mathcal{O}_{X}\left( V\right) $ for $j=1,2.$

Repeat the construction in the previous subsection \S 3.2 for the set $%
\Delta $ and the Galois group $G.$ Then we obtain an arithmetic variety $%
Z_{3}$.

By \emph{Lemma 3.5} it is clear that $Z_{3}/X,$ $Z_{3}/Z_{1},$ and $%
Z_{3}/Z_{2}$ are all quasi-galois closed. Hence, $Z_{3}\in X_{qc}\left[
\Omega \right] .$
\end{proof}

\begin{lemma}
Let $Z_{1},Z_{2},Z_{3}\in X_{qc}\left[ \Omega \right] $ be such that $%
Z_{1}/Z_{2}$ and $Z_{2}/Z_{3}$ are both quasi-galois closed. Then $%
Z_{1}/Z_{3}$ is quasi-galois closed. \
\end{lemma}

\begin{proof}
Let $f_{1}:Z_{1}\rightarrow Z_{2}$ and $f_{2}:Z_{2}\rightarrow
Z_{3} $ be surjective morphisms of finite types by which $Z_{1}/Z_{2}$ and $%
Z_{2}/Z_{3}$ are both quasi-galois closed. Put $f_{3}=f_{2}\circ f_{1}.$

Then $f_{3}:Z_{1}\rightarrow Z_{3}$ is a surjective morphism of finite type.
It is seen that $f_{3}$ is affine since $f_{1}$ and $f_{2}$ are both affine.
By \emph{Lemma 3.5} it follows that $Z_{1}/Z_{3}$ is quasi-galois closed by $%
f_{3}$.
\end{proof}

\subsection{A property for qc integral varieties}

There is a result for integral varieties which is similar to arithmetic
varieties (see Lemma 3.7 above or \cite{An2}).

\begin{lemma}
Let $f:X\rightarrow Y$ be a surjective morphism of integral varieties.
Suppose that $X/Y$ is quasi-galois closed by $f$ and that $k\left( X\right) $
is canonically a Galois extension over $k\left( Y\right) .$ Then $f$ is
affine and there is a group isomorphism
\begin{equation*}
Aut\left( X/Y\right) \cong Gal\left( k\left( X\right) /k\left( Y\right)
\right) .
\end{equation*}
\end{lemma}

\begin{proof}
By \emph{Lemma 3.8}, without loss of generality, assume that $X$ and $Y$
have reduced affine coverings $\mathcal{C}_{X}$ and $\mathcal{C}_{Y}$ with
values in $\Omega =k\left( X\right) ^{al},$ respectively.

Let $$\Delta \subseteq k\left( X\right) \setminus k\left( Y\right) $$ be the
set of generators of $k\left( X\right) $ over $k\left( Y\right) .$ Put $$
G=Gal\left( k\left( X\right) /k\left( Y\right) \right) .$$

Repeat the construction for a geometric model in \S 3.2, we obtain an
integral variety $X^{\prime }$ such that $A_{\alpha }$ has only one
conjugate over $B_{\alpha }$ for any $(U_{\alpha },\phi _{\alpha };A_{\alpha
})\in \mathcal{C}_{X}$ and for any affine open set $V_{\alpha }$ in $Y$ with
$U_{\alpha }\subseteq \varphi ^{-1}(V_{\alpha })$, where $B_{\alpha }=%
\mathcal{O}_{Y}(V_{\alpha })$.

By \emph{Lemma 3.5} it is seen that $X^{\prime }/Y$ is quasi-galois closed.
It follows that $X^{\prime }=X$ from \emph{Remark 1.2}. Hence, $f$ must be
affine.

Now define a mapping
\begin{equation*}
t:Aut\left( X/Y\right) \longrightarrow Gal\left( k\left( X\right) /k\left(
Y\right) \right)
\end{equation*}%
by
\begin{equation*}
\sigma =(\sigma ,\sigma ^{\sharp })\longmapsto t(\sigma )=\left\langle
\sigma ,\sigma ^{\sharp -1}\right\rangle
\end{equation*}%
where $\left\langle \sigma ,\sigma ^{\sharp -1}\right\rangle $ is the map of
$k(X)$ into $k(X)$ given by
\begin{equation*}
\left( U,f\right) \in \mathcal{O}_{X}(U)\subseteq k\left( X\right)
\longmapsto \left( \sigma \left( U\right) ,\sigma ^{\sharp -1}\left(
f\right) \right) \in \mathcal{O}_{X}(\sigma (U))\subseteq k\left( X\right)
\end{equation*}%
for any $f\in \mathcal{O}_{X}|_{U}(U)$, where $U$ runs through all open sets
in $X.$

It is easily seen that $t$ is well-defined.

As we have done for the proof of the main theorem in \cite{An2}, in the
following we will proceed in several steps to prove that $t$ is a group
isomorphism.

\emph{Step 1.} Prove that ${t}$ is injective. Take any $\sigma ,\sigma
^{\prime }\in {Aut}\left( X/Y\right) $ such that $t\left( \sigma \right)
=t\left( \sigma ^{\prime }\right) .$ We have
\begin{equation*}
\left( \sigma \left( U\right) ,\sigma ^{\sharp -1}\left( f\right) \right)
=\left( \sigma ^{\prime }\left( U\right) ,\sigma ^{\prime \sharp -1}\left(
f\right) \right)
\end{equation*}%
for any $\left( U,f\right) \in k\left( X\right) .$ In particular, for any $%
f\in \mathcal{O}_{X}(U_{0})$ we have
\begin{equation*}
\left( \sigma \left( U_{0}\right) ,\sigma ^{\sharp -1}\left( f\right)
\right) =\left( \sigma ^{\prime }\left( U_{0}\right) ,\sigma ^{\prime \sharp
-1}\left( f\right) \right)
\end{equation*}%
where $U_{0}$ is an affine open subset of $X$ such that $\sigma \left(
U_{0}\right) $ and $\sigma ^{\prime }\left( U_{0}\right) $ are both
contained in $\sigma \left( U\right) \cap \sigma ^{\prime }\left( U\right) $.

It is seen that
\begin{equation*}
\sigma |_{U_{0}}=\sigma ^{\prime }|_{U_{0}}
\end{equation*}%
holds as isomorphisms of schemes. As $U_{0}$ is dense in $X$, we have
\begin{equation*}
\sigma =\sigma |_{\overline{U_{0}}}=\sigma ^{\prime }|_{\overline{U_{0}}%
}=\sigma ^{\prime }
\end{equation*}%
on the whole of $X$; then $$\sigma \left( U\right) =\sigma ^{\prime }\left(
U\right) ;$$ it follows that $$\sigma =\sigma ^{\prime }$$ holds.

\emph{Step 2.} Prove that ${t}$ is surjective. Fixed any element $\rho $ of
the group $Gal\left( k\left( X\right) /k\left( Y\right) \right) $.

As $k(X)=\{(U_{f},f):f\in \mathcal{O}_{X}(U_{f})\text{ and }U_{f}\subseteq X%
\text{ is open}\}$, we have
\begin{equation*}
\rho :\left( U_{f},f\right) \in k\left( X\right) \longmapsto \left( U_{\rho
\left( f\right) },\rho \left( f\right) \right) \in k\left( X\right) ,
\end{equation*}
where $U_{f}$ and $U_{\rho (f)}$ are open sets in $X$, $f$ is contained in $%
\mathcal{O}_{X}(U_{f})$, and $\rho (f)$ is contained in $\mathcal{O}%
_{X}(U_{\rho (f)})$.

In fact, we prove that each element of $Gal\left( k\left( X\right) /k\left(
Y\right) \right) $ give us a unique element of ${Aut}(X/Y)$:

Fixed any affine open set $V$ of $Y$. It is easily seen that for each affine
open set $U\subseteq \phi ^{-1}(V)$ there is an affine open set $U_{\rho }$
in $X$ such that $\rho $ determines an isomorphism $\lambda _{U}$ between
affine schemes $(U,\mathcal{O}_{X}|_{U})$ and $(U_{\rho },\mathcal{O}%
_{X}|_{U_{\rho }})$.

Take any affine open sets $V\subseteq Y$. It is seen that
\begin{equation*}
\lambda _{U}|_{U\cap U^{\prime }}=\lambda _{U^{\prime }}|_{U\cap U^{\prime }}
\end{equation*}%
holds as morphisms of schemes for any affine open sets $U,U^{\prime
}\subseteq \phi ^{-1}(V)$.

By gluing $\lambda _{U}$ along all such affine open subsets $U$, we have an
automorphism $\lambda $ of the scheme $X$ such that $\lambda |_{U}=\lambda
_{U}$ for any affine open set $U$ in $X$. It is easily seen that $t\left(
\lambda \right) =\rho $. Hence, ${t}$ is surjective. This completes the
proof.
\end{proof}

\begin{remark}
Let $X$ and $Y$ be integral varieties and let $X$ be quasi-galois closed
over $Y$ by a surjective morphism $\phi $. Then there is a natural
isomorphism
\begin{equation*}
\mathcal{O}_{Y}\cong \phi _{\ast }(\mathcal{O}_{X})^{{Aut}\left( X/Y\right) }
\end{equation*}%
where $(\mathcal{O}_{X})^{{Aut}\left( X/Y\right) }(U)$ denotes the invariant
subring of $\mathcal{O}_{X}(U)$ under the natural action of ${Aut}\left(
X/Y\right) $ for any open subset $U$ of $X$.
\end{remark}

\subsection{A universal cover for the qc fundamental group}

Let $X$ be an arithmetic variety and let $\Omega $ be an algebraically
closed field such that $k\left( X\right) ^{al}\subseteq \Omega .$ Without
loss of generality, assume that $X$ has a reduced affine covering with
values in $k\left( X\right) ^{al}.$

Put $$G=Gal\left( \Omega /k\left( X\right) \right) $$ and take a set $$\Delta
\subseteq \Omega \setminus k\left( X\right) $$ of generators of $\Omega $
over $k\left( X\right) .$ Repeat the construction in \S 3.2, we have an
integral variety $X_{\Omega }$ such as in the following lemma.

\begin{lemma}
Let $X$ be an arithmetic variety and let $\Omega $ be an algebraically
closed field containing $k\left( X\right) ^{al}.$ Then there is an integral
variety $X_{\Omega }$ and a surjective morphism $f_{\Omega }:X_{\Omega
}\rightarrow X$ satisfying the conditions:

\begin{itemize}
\item $k\left( X_{\Omega }\right) =\Omega $;

\item $f_{\Omega }$ is affine;

\item $k\left( X_{\Omega }\right) $ is Galois over $k\left( X\right) ;$

\item $X_{\Omega }/X$ is quasi-galois closed by $f_{\Omega }$.
\end{itemize}
\end{lemma}

Such an integral variety $X_{\Omega }$ is called a \textbf{geometric model}
for the qc fundamental group $\pi _{1}^{qc}\left( X;\Omega \right)$ or a
\textbf{universal cover} over $X$ for the group $\pi _{1}^{qc}\left(
X;\Omega \right)$.

\subsection{A universal cover for the \'{e}tale fundamental group}

Let $X$ be an arithmetic variety and let $s$ be a geometric point of $X$
over $k\left( X\right) ^{al}.$ Without loss of generality, assume that $X$
has a reduced affine covering with values in $k\left( X\right) ^{al}.$

Put $$
\Omega _{et}=k\left( X\right) ^{un};$$
$$G=Gal\left( \Omega _{et}/k\left( X\right) \right) .$$ Take a set $$\Delta
\subseteq \Omega _{et}\setminus k\left( X\right)$$ of generators
of $\Omega _{et}$ over $k\left( X\right) .$ Repeat the construction in \S %
3.2, we have an integral variety $X_{\Omega _{et}}$ such as in the following
lemma.

\begin{lemma}
Let $X$ be an arithmetic variety and let $s$ be a geometric point of $X$
over $k\left( X\right) ^{al}.$ Then there is an integral variety $X_{\Omega
_{et}}$ and a surjective morphism $f_{\Omega _{et}}:X_{\Omega
_{et}}\rightarrow X$ satisfying the conditions:

\begin{itemize}
\item $k\left( X_{\Omega _{et}}\right) =\Omega _{et}$;

\item $f_{\Omega _{et}}$ is affine;

\item $k\left( X_{\Omega _{et}}\right) $ is Galois over $k\left( X\right) ;$

\item $X_{\Omega _{et}}/X$ is quasi-galois closed by $f_{\Omega _{et}}$.
\end{itemize}
\end{lemma}

Such an integral variety $X_{\Omega _{et}}$ is called a \textbf{geometric
model} for the \'{e}tale fundamental group $\pi _{1}^{et}\left( X;s\right) $ or
a \textbf{universal cover} over $X$ for the group $\pi _{1}^{et}\left(
X;s\right) $.

\subsection{Proof of the main theorem}

Now we can prove the main theorem of the paper.

\begin{proof}
(\textbf{Proof of Theorem 2.2}) We will proceed in several steps.

\emph{Step 1.} By \emph{Lemma 3.13} we have
\begin{equation*}
\begin{array}{l}
Gal\left( \Omega /k\left( X\right) \right)\\

 \cong {\lim_{\longleftarrow }}%
_{Z\in X_{qc}\left[ \Omega \right] }{Gal\left( k\left( Z\right) /k\left(
X\right) \right) }\\

\cong {\lim_{\longleftarrow }}_{Z\in X_{qc}\left[ \Omega \right] }{Aut\left( Z/X\right) }\\

={\pi }_{1}^{qc}\left( X;\Omega \right) .
\end{array}
\end{equation*}

\emph{Step 2.} Let $X_{et}\left[ s\right] $ be the set of finite \'{e}tale
Galois covers of $X$ with respect to the geometric point $s$ (see ). For any $%
Z_{1},Z_{2}\in X_{et}\left[ s\right] $ we say $$Z_{1}\leq Z_{2}$$ if and only
if $Z_{2}$ is a finite \'{e}tale Galois cover of $Z_{1}.$ Then $X_{et}\left[ s%
\right] $ is a partially ordered set. Put
$$X_{qc}\left[ \Omega ;s\right] =X_{qc}\left[ \Omega \right] \cap X_{et}%
\left[ s\right] .$$

Let $Z_{1},Z_{2}\in X_{qc}\left[ \Omega ;s\right] .$ From \emph{Lemma 3.5}
it is easily seen that $Z_{2}/Z_{1}$ is a finite \'{e}tale Galois cover if and
only if $Z_{2}/Z_{1}$ is quasi-galois closed. Hence, the two partial orders $\leq $ in $%
X_{qc}\left[ \Omega ;s\right] $ coincide with each other.

It follows that $X_{qc}\left[ \Omega ;s\right] $ and $X_{et}\left[ s%
\right] $ are cofinal directed sets according to the construction in \S\ 3.2.

\emph{Step 3.} By the construction in \S\ 3.2 again, for the universal
covers we have $$X_{\Omega }={\lim_{\rightarrow }}_{Z\in X_{qc}\left[ \Omega %
\right] }Z$$ and $$X_{\Omega _{et}}={\lim_{\rightarrow }}_{Z\in X_{et}\left[ s%
\right] }Z$$ as direct limits.

From \emph{Lemma 3.14} we have
\begin{equation*}
\begin{array}{l}
Gal\left( \Omega _{et}/k\left(
X\right) \right) \\

\cong {\lim_{\longleftarrow }}_{Z\in X_{et}\left[ s\right] }{Gal\left(
k\left( Z\right) /k\left( X\right) \right) }\\

\cong {\lim_{\longleftarrow }}_{Z\in X_{et}\left[ s\right] }{Aut\left(
Z/X\right) }\\

\cong {\pi }_{1}^{et}\left( X;s\right) .
\end{array}
\end{equation*}

On the other hand, by \emph{Step 2} we have
\begin{equation*}
\begin{array}{l}
{\lim_{\longleftarrow }}_{Z\in X_{et}\left[ s\right] }{Gal\left( k\left(
Z\right) /k\left( X\right) \right) }\\

\cong Gal\left( \Omega _{et}/k\left( X\right) \right) \\

\cong {\lim_{\longleftarrow }}_{Z\in X_{qc}\left[ \Omega ;s\right] }{%
Gal\left( k\left( Z\right) /k\left( X\right) \right) .}
\end{array}
\end{equation*}

Hence,
$${\pi }_{1}^{et}\left( X;s\right) \cong Gal\left( \Omega _{et}/k\left(
X\right) \right) .$$

It follows that ${\pi }_{1}^{et}\left( X;s\right) $ is a subgroup of ${\pi }%
_{1}^{qc}\left( X;\Omega \right) .$
It is evident that ${\pi }_{1}^{et}\left( X;s\right) $ is normal in the group ${\pi }%
_{1}^{qc}\left( X;\Omega \right) $ since $k\left( X_{\Omega _{et}}\right)$ is a subfield of
$k\left( X_{\Omega }\right) $ and $k\left( X_{\Omega _{et}}\right) $
is Galois over $k\left( X\right) $.

This completes the proof.
\end{proof}

\end{document}